\documentclass[reqno]{amsart}

\usepackage{amsrefs}
\usepackage{hyperref}
\usepackage{bbm}

\newtheorem{theorem}{Theorem}
\newtheorem*{theorem*}{Theorem}
\newtheorem{lemma}[theorem]{Lemma}
\newtheorem{conjecture}[theorem]{Conjecture}
\newtheorem{proposition}[theorem]{Proposition}

\newtheorem{claim}[theorem]{Claim}

\newtheorem{question}[theorem]{Question}

\theoremstyle{exercise}

\theoremstyle{definition}

\usepackage{graphicx}
\usepackage{pstricks, enumerate, pst-node, pst-text, pst-plot}

\numberwithin{equation}{section}
\numberwithin{theorem}{section}

\DeclareMathOperator{\sgn}{sgn}
\DeclareMathOperator{\aut}{Aut}

\newcommand{\argmin}{\operatornamewithlimits{argmin}}

\renewcommand{\P}[1]{{\mathbb{P}}\left[{#1}\right]}
\newcommand{\CondP}[2]{{\mathbb{P}}\left[{#1}\middle\vert{#2}\right]}

\newcommand{\ind}[1]{\mathbbm{1}_{#1}}

\newcommand{\half}{{\textstyle \frac12}}

\newcommand{\fourth}{{\textstyle \frac14}}

\newcommand{\N}{\mathbb N} \newcommand{\R}{\mathbb R}

\newcommand{\eps}{\epsilon}

\newcommand{\op}{A}
\newcommand{\limop}{Z}
\newcommand{\Sest}{\hat{S}}
\newcommand{\bdd}{\mathcal{B}}

\newcommand{\moment}[2]{{\mbox{M}_{#1}}\left[{#2}\right]}
\newcommand{\dist}{\rho}
\newcommand{\cc}{C}

\newcommand{\sigmazero}{\mathcal{A}_0}
\newcommand{\sigmainf}{\mathcal{A}_\infty}
\newcommand{\neigh}[1]{N({#1})}

\begin{document}

\title[Majority Dynamics and the Retention of Information]{Majority
  Dynamics and the Retention of Information}

\author[Omer Tamuz]{Omer Tamuz}
\address[O.~Tamuz]{Weizmann Institute of Science, Faculty of Mathematics and Computer Science, Rehovot, Israel.}
\email[O.~Tamuz]{omer.tamuz@weizmann.ac.il}
\thanks{ Omer Tamuz is supported by ISF grant 1300/08, and is a
  recipient of the Google Europe Fellowship in Social Computing, and
  this research is supported in part by this Google Fellowship.}

\author[Ran J.\ Tessler]{Ran J.\ Tessler}
\address[R.J.~Tessler]{Hebrew University, Einstein Institute of Mathematics, Jerusalem, Israel.}
\email[R.J.~Tessler]{ran.tessler@mail.huji.ac.il}

\subjclass[2010]{Primary:. Secondary:}

\date{\today}


\begin{abstract}
  We consider a group of agents connected by a social network who
  participate in {\em majority dynamics}: each agent starts with an
  opinion in $\{-1,+1\}$ and repeatedly updates it to match the
  opinion of the majority of its neighbors.

  We assume that one of $\{-1,+1\}$ is the ``correct'' opinion $S$,
  and consider a setting in which the initial opinions are independent
  conditioned on $S$, and biased towards it. They hence contain enough
  information to reconstruct $S$ with high probability. We ask whether
  it is still possible to reconstruct $S$ from the agents' opinions
  after many rounds of updates.

  While this is not the case in general, we show that indeed, for a
  large family of bounded degree graphs, information on $S$ is
  retained by the process of majority dynamics.

  Our proof technique yields novel combinatorial results on majority
  dynamics on both finite and infinite graphs, with applications to
  zero temperature Ising models.
\end{abstract}

\maketitle
\tableofcontents

\section{Introduction}
Consider a group of people (agents) who each carry one of two possible
opinions regarding some issue. Each agent forms an initial opinion,
and then repeatedly updates it by observing the opinions of its
neighbors in a social network.

Many variants of this model have been studied in diverse settings and
disciplines such as Economics
(e.g.~\cite{ellison1993rules,bala1998learning}), Statistical Mechanics
(e.g.~\cite{howard2000zero, fontes2002stretched, de2003convergence}),
Computer Science (see a survey by Shah~\cite{shah2009gossip}), and
Mathematics (e.g.~\cite{kanoria2011majority,
  camia2002clusters}). These include variations on how the agents
acquire their initial opinions (e.g., deterministically, at random,
arbitrarily or through some other process), what they aim to achieve
in this process (e.g., rational agents in
Economics~\cite{gale2003bayesian, bala1998learning,
  mossel2012strategic}, message passing agents in Computer
Science~\cite{kempe2003gossip}), and how they go about updating their
opinions in order to achieve this.

A particularly well studied model is that of {\em majority dynamics},
in which agents update their opinions to match that of the majority of
their neighbors.  We choose a setting in which each agent's initial
opinion carries some independent information regarding the ``true''
opinion, and where it is possible to discover this truth with high
probability by aggregating the initial opinions.  The question we
tackle is the following: is information lost through the process of
majority dynamics? When can the ``truth'' be well estimated even after
people have exchanged opinions?

These questions were, to our knowledge, first considered in the
context of majority dynamics in Mossel, Neeman and
Tamuz~\cite{mossel2012majority}, who gave both positive and negative
examples. We extend their work in several directions, as described
below.

As part of our analysis we develop a combinatorial tool that is
instrumental in studying majority dynamics. Using it, we present a
number of purely combinatorial results on this process. These results
can be applied directly to what is known as the dynamics of the zero
temperature translation-invariant ferromagnetic Ising model, which, on
odd degree graphs, is the same as what we below call asynchronous
majority dynamics. In particular, Theorem~\ref{thm:opinion-changes}
gives, for every odd degree lattice of (say) polynomial growth, a
constant $C$ such that regardless of the starting configuration, with
probability one no site changes its state more than $C$ times.

\subsection{Definitions and a statement of the problem}
Let $V$ be a finite or countably infinite set of agents.  Let
$G=(V,E)$, the {\em social network}, be an undirected, connected,
locally finite graph. We denote the neighbors of $i \in V$ by
$\neigh{i} = \{j \,:\, (i,j) \in E\}$, and say that $G$ is
$d$-bounded degree when $|\neigh{i}| \leq d$ for all $i \in V$.

We denote by $\op^i_t \in \{-1,+1\}$ agent $i$'s {\em opinion} at
time $t$. After drawing some initial opinions $\{\op^i_0\}_{i \in
  V}$ from a distribution we describe below, the agents proceed to
update their opinions using {\em majority dynamics}.

We consider two version of majority dynamics, namely a synchronous and
an asynchronous one, and prove all our results for both. In the
synchronous model, $t$ takes values in the non-negative integers only,
and we set
\begin{align}
  \label{eq:maj-synch}
  \op^i_{t+1} = \sgn \sum_{j \in \neigh{i}}\op^j_t.
\end{align}
In the asynchronous model, we equip each agent with an independent,
unit rate Poisson clock, and let each agent update its opinions at the
times of its clock rings, using
\begin{align}
  \label{eq:maj-asynch}
  \op^i_t = \sgn \sum_{j \in \neigh{i}}\op^j_t.
\end{align}
We assume throughout that the degree of every vertex is odd, so that
there are no ``ties'' and $\op^i_t \in \{-1,+1\}$ for all
$t$. Alternatively, given a graph with even degrees, one could add $i$
or remove $i$ from $\neigh{i}$ to make $|\neigh{i}|$ odd.

In the asynchronous model, in order to ensure that the model is well
defined, we must rule out the possibility that there will occur an
infinite sequence of clock ticks $t_1 > t_2 > \cdots > 0$ such
that agent $i$ updates at time $t_i$ and $i \in \neigh{i+1}$. To
ensure that this occurs with probability zero, it is sufficient to
assume that degrees are bounded. We indeed make this assumption
throughout, and prove that it is sufficient in
Claim~\ref{clm:asynch-well-defined} below.

A classical result~\cite{mcculloch1943logical} is that for finite
graphs, in the synchronous model, each agent's opinion either
converges, or, from some time on, oscillates between $-1$ and $+1$
with period two; that is, each agent's opinion eventually has period
at most two. In the asynchronous model all opinions converge for
finite graphs.

The same can be shown to hold for bounded degree infinite graphs with
sufficiently slow growth. To state this result we shall need some
definitions. Denote graph distances by $\dist(\cdot,\cdot)$, let
$n_r(G,i) = |\{j\,:\,\dist(i,j)=r\}|$ be number of vertices at graph
distance $r$ from $i$ in $G$. Finally, denote
\begin{align*}
  \moment{d}{G,i} = \sum_{r=0}^\infty\left(\frac{d+1}{d-1}\right)^{-r}n_r(G,i);
\end{align*}
the usefulness of this definition will be demonstrated below.

Ginosar and Holzman~\cite{ginosar2000majority} prove the following
result, which is a strengthening of a somewhat weaker claim by
Moran~\cite{moran1995period}.
\begin{theorem}[Moran, Ginosar and Holzman\footnote{Note that Moran,
    as well as Ginosar and Holzman, prove their theorems for the
    synchronous model; we extend it to the asynchronous model.}]
  \label{thm:moran}
  Let $G$ be a $d$-bounded degree graph such that $\moment{d}{G,i} <
  \infty$ for some ($\leftrightarrow$ all) $i \in V$.  Then each
  agent's opinion eventually has period at most two in the synchronous
  model, and converges in the asynchronous model.
\end{theorem}
Note that this is a combinatorial (rather than a probabilistic)
result, in the sense that it holds for every set of initial opinions
$\{\op^i_0\}_{i \in V}$, and, in the asynchronous model, for every
choice of clock ring times (in which there are no rings at the same
time). We prove this result using a novel combinatorial tool, which
yields additional insights into this process (see
Section~\ref{sec:maj-dyn-comb}). In particular, we bound the number of
times that an agent may change its opinion, and show that an agent
will never change its opinion if a large enough neighborhood around it
agrees with it.

In Section~\ref{sec:counterexample} we provide an example of a
$d$-regular graph for which $\moment{d}{G,i} = \infty$ for all $i \in
V$, and for which, in the synchronous model, each agent's opinion
eventually has period at most two for any initial set of
opinions. This shows that the sufficient condition of
Theorem~\ref{thm:moran} is not necessary. Indeed, the existence of a
simple, geometrical, necessary and sufficient condition for
convergence to period at most two is an interesting open question.

It should be noted that had we allowed even degrees with random tie
breaking, the dynamics would have changed dramatically. Tessler and
Louidor~\cite{louidortessler2010} discuss the asynchronous model for
trees with even degrees larger than two (for which it may be that
$\moment{d}{G,i} < \infty$). It is shown there that if the initial
opinions are chosen with i.i.d symmetric Bernoulli distribution, then
almost surely there exist some agents that change their opinions
infinitely many times.

We henceforth consider only {\em slow growth} graphs, i.e.\ graph for
which
\[
\moment{d}{G,i} < \infty,
\]
and for which, by Theorem~\ref{thm:moran}, we can define random
variables $\limop^i$ taking values in $\{-1,+1\}$ and given by
\begin{align*}
  \limop^i = \lim_{t \to \infty}\op^i_{2t},
\end{align*}
where the limit is taken over $t \in \N$.

We next describe how the agents acquire their initial opinions.  Let
$S \in \{-1,+1\}$ be the {\em state of the world} with $\P{S=-1} =
\P{S=1} = \half$. We think of $\op^i_t \in \{-1,+1\}$ as agent $i$'s
{\em opinion regarding $S$} at time $t$, and draw $\{\op^i_0\}_{i \in
  V}$ as follows: We fix some $\half < p < 1$, and let $\op^i_0=S$
with probability $p$ and $\op^i_0=-S$ with probability $1-p$, with the
events `$\op^i_0=S$' being independent of each other and of $S$. Note
that it follows that the random variables $\op^i_0$ are not
independent, but are independent (in fact, i.i.d) conditioned on $S$.

When $|V|$ is finite but large, then $S$ can be estimated with high
probability given $\{\op^i_0\}_{i \in V}$. When $|V|$ is infinite, $S$
can be estimated exactly given $\{\op^i_0\}_{i \in V}$. Formally, let
$|V| = n$, and let
\begin{align*}
  \Sest_0 = \sgn \sum_{i \in V}\op^i_0.
\end{align*}
Then by the Chernoff bound
\begin{align*}
  \P{\Sest_0 \neq S} \leq e^{-O(n)}.
\end{align*}
For the case that $|V|=\infty$, let
\begin{align*}
  \sigmazero = \sigma\left(\left\{\op^i_0\right\}_{i \in V}\right).
\end{align*}
Here we denote by $\sigma(X)$ the sigma-algebra generated by $X$, a
random variable\footnote{This is the smallest sigma-algebra for which
  $X$ is measurable.}.  Then it holds that
\begin{align*}
  \inf_{\Sest_0 \in \sigmazero}\P{\Sest_0 \neq S} = 0,
\end{align*}
where, by a slight abuse of notation, we say that $\Sest_0 \in
\sigmazero$ when $\Sest_0$ is $\sigmazero$-measurable.

The question that we tackle is the following: when is it the case that
$S$ can still be reconstructed from the limiting opinions
$\{\limop^i\}_{i \in V}$?  Formally, define $\sigmainf =
\sigma\left(\left\{\limop^i\right\}_{i \in V}\right)$,
and let the probability of error (in the reconstruction of $S$ using
$\{\limop^i\}_{i \in V}$) be given by
\begin{align*}
  \delta(G,p) =   \inf_{\Sest \in \sigmainf}\P{\Sest \neq S},
\end{align*}
where, to remind the reader, $p = \P{\op^i_0=S}$ is a parameter of our
measure $\mathbb{P}$.

We ask the question of whether $\delta(G,p)$ is equal to zero in the
case of an infinite graph, or is close to zero in the case of a large
finite graph.  Formally, let $\{G_n\}_{n \in \N}$ be a sequence of
finite graphs such that $\lim_n|V_n| = \infty$. For which sequences
and values of $p$ is it the case that $\lim_n\delta(G_n,p)=0$? And for
which infinite $G$ does it hold that $\delta(G,p)=0$?

Berger~\cite{berger2001dynamic} gives an example of a sequence of
finite graphs with $\lim_n|V_n| = \infty$ and such that in each graph
there exists a {\em dynamic monopoly} of size eighteen: a set $W
\subset V$ of eighteen vertices with the property that if
$\{\op^j_0\}_{j \in W}$ are all equal to some $s \in \{-1,+1\}$ then
$\limop^i=s$ for all $i \in V$ (in the synchronous model)\footnote{We
  slightly weaken his definition; he requires that $\op^i_t=s$ for all
  $i$ in some finite time $t$.}.  Since $\P{\op^j_0 = -S\quad\forall j
  \in W} = (1-p)^{18}$, it follows that $\delta(G_n,p) \geq
(1-p)^{18}$.

This example involves graphs with increasingly large degrees.  We
offer the following conjecture (see also~\cite{mossel2012majority}).
\begin{conjecture}
  Let $\{G_n\}_{n \in \N}$ be a sequence of $d$-bounded degree finite
  graphs such that $\lim_n|V_n| = \infty$. Then
  \begin{align*}
    \lim_{n \to \infty}\delta(G_n,p)=0
  \end{align*}
  for all $\half < p < 1$.
\end{conjecture}

Mossel, Neeman and Tamuz~\cite{mossel2012majority} show that if
$\{G_n\}$ is sequence of {\em transitive} graphs (that is, graphs in
which all nodes play the same role, see Section~\ref{sec:random}
below) then $\lim_n\delta(G_n,p) = 0$. In fact, they show that the
same holds under a weaker assumption, namely that each geometric
equivalence class is large (again, see Section~\ref{sec:random} for a
precise definition). They also show the same for good enough {\em
  expander} graphs.

For infinite graphs $G$, our question is whether $\delta(G,p) =
0$. While we provide some positive results below, we do not know the
answer to the following seemingly basic question:
\begin{question}
  Does there exist an infinite graph $G$ and $\half < p < 1$ such that
  $\delta(G,p) > 0$?
\end{question}
We prove the following claim\footnote{Although the proof of this claim
  is rather straightforward, we have not found it in the literature.}
(see~\cite{peleg1998size} for related work).
\begin{theorem}
  \label{thm:no-dynamic-monopoly}
  In the synchronous model, no infinite, locally finite graph has a
  dynamic monopoly of finite size.
\end{theorem}
This result suggests that perhaps $\delta(G,p) = 0$ for every finite
degree, slow growth graph $G$.

\subsection{Results}
\subsubsection{Majority Dynamics}
\label{sec:maj-dyn-comb}
We begin by proving two combinatorial claims regarding majority
dynamics. These claims may be of independent interest, but are also
useful in proving our main results.

The first result is a quantitative version of Theorem~\ref{thm:moran}.
\begin{theorem}
  \label{thm:opinion-changes}
  Let $G$ be a $d$-bounded degree graph such that $\moment{d}{G,i} <
  \infty$ for some ($\leftrightarrow$ all) $i \in V$. Then
  \begin{enumerate}
  \item In the synchronous model, the number of times $t$ for which
    $\op^i_{t+1} \neq \op^i_{t-1}$ is at most $ \frac{d+1}{d-1} \cdot
    d \cdot \moment{d}{G,i}$.
  \item In the asynchronous model, the number of times in which $i$
    changes its opinion is at most $ \frac{d+1}{d-1} \cdot 2d \cdot
    \moment{d}{G,i}$.
  \end{enumerate}
\end{theorem}
We would like to thank the anonymous referee for pointing out to us
that this theorem can, for the synchronous model, be easily derived
from the work of Ginosar and Holzman.

The next result shows that if a sufficiently large neighborhood of $i$
starts with a certain opinion then $i$ will always have this opinion.
\begin{theorem}
  \label{thm:bunkers}
  Let $G$ be a $d$-bounded degree graph such that $\moment{d}{G,i} <
  \infty$ for some ($\leftrightarrow$ all) $i \in V$. Let $r_0$ be
  such that
  \begin{align*}
    \frac{d+1}{d-1} \cdot 2d \cdot
    \sum_{r>r_0}\left(\frac{d+1}{d-1}\right)^{-r}n_r(G,i) < 1.
  \end{align*}

  If $\op^j_T=\op^i_T$ for some $T$ and for all $j$ such that
  $\dist(i,j) \leq r_0+2$, then $\op^i_t=\op^i_T$ for all $t > T$.
\end{theorem}

\subsubsection{Invariant random subgraphs}
\label{sec:random}
A {\em graph automorphism} of $G=(V,E)$ is a bijection $h : V \to V$
such that $(i,j) \in E \leftrightarrow (h(i), h(j)) \in E$. We denote
by $\aut(G)$ the automorphism group of $G$. Indeed, it is easy to
verify that the graph automorphisms of a given graph form a group
under composition.

Let $H$ be a subgroup of $\aut(G)$. We say that $H$ acts {\em
  transitively} on (the vertices of) $G$ when, for every $i,j \in V$
there exists an $h \in H$ such that $h(i) = j$. Equivalently, $H$ acts
transitively on $G$ when $V/H$, the set of $H$ orbits of $V$, is a
singleton. We say that $H$ acts {\em quasi-transitively} on $G$ when
$V/H$ is finite. Finally, we say that $G$ is (quasi-) transitive when
$\aut(G)$ acts on it (quasi-) transitively.

A subgraph of $G_0=(V_0,E_0)$ is a graph $G=(V,E)$ such that $V
\subseteq V_0$ and $E \subseteq E_0$ is a set of edges on $V$.  Let
$H$ be a subgroup of $\aut(G_0)$. A random $G_0$-subgraph $G$ (that is,
a random variable $G$ that takes values in the space of subgraphs of a
graph $G_0$) is said to have an $H$-{\em invariant} distribution if
for all $h \in H$ the law of $h(G)$ equals the law of $G$.

Note that $G$ could have nodes with even degrees even when $G_0$ has
odd degrees. Since we will want to apply majority dynamics to $G$, we
add or remove self-loops to $G$ in order to make all degrees odd. This
does not affect the fact that $G$ is $H$-invariant; the modified $G$
is $H$-invariant iff the unmodified $G$ was $H$-invariant. This also
does not increase any degree beyond what it was on $G_0$, since we do
not add a self-loop unless we remove another edge.

The following is our main result for this section.
\begin{theorem}
  \label{thm:percolation-learning}
  Let $G_0$ be a quasi-transitive infinite graph with maximal degree
  $d$ such that $\moment{d}{G_0,i} < \infty$ for some
  ($\leftrightarrow$ all) $i \in V_0$.

  Let $H \leq \aut(G_0)$ act quasi-transitively on $G_0$, and let $G$
  be an infinite connected random subgraph of $G_0$ with an
  $H$-invariant distribution. Then
  \begin{align*}
    \delta(G,p) = 0
  \end{align*}
  almost surely for any $\half < p < 1$.
\end{theorem}
It follows from Theorem~\ref{thm:percolation-learning} that
$\delta(G,p)=0$ for any quasi-transitive, slow growth infinite graph
$G$ and any $\half < p < 1$. This is already a non-trivial result,
which is generalized in Theorem~\ref{thm:percolation-learning}.

\subsubsection{Families of uniformly bounded growth graphs}
Given $d \geq 3$ and a function $f : \N \to \N$, let $\bdd(f,d)$ be the
family of $d$-bounded degree graphs $G$ such that $n_r(G,i) \leq f(r)$
for all vertices $i$ in $G$ and $r \in \N$.

Reusing the notation $\moment{d}{\cdot}$, let
\begin{align*}
  \moment{d}{f} = \sum_{r=0}^\infty\left(\frac{d+1}{d-1}\right)^{-r} \cdot f(r).
\end{align*}
We say that $f : \N \to \N$ has {\em slow growth} if $\moment{d}{f} < \infty$.

\begin{theorem}
  \label{thm:bdd-retention}
  Fix an odd $d \geq 3$, and let $f : \N \to \N$ have slow growth.
  There exists a $\half < p_0 < 1$ such that for all $p_0 < p < 1$ and
  for all sequences of finite graphs $\{G_n\}_{n \in \N}$ in
  $\bdd(f,d)$ such that $\lim_n|V_n| = \infty$ it holds that
  \begin{align*}
    \lim_n\delta(G_n,p) = 0,
  \end{align*}
  and for all infinite graphs $G \in \bdd(f,d)$ it holds that
  \begin{align*}
    \delta(G,p) = 0.
  \end{align*}
\end{theorem}

\section{Acknowledgments}
We would like to thank Elchanan Mossel for useful discussions and
ideas, and the anonymous referee for additional important comments and
corrections.

\section{Proofs}
\subsection{The Lyapunov functional}
Following Goles and Olivos~\cite{goles1980periodic}, we analyze the
process of Majority dynamics using the technique of Lyapunov
functionals. In particular, we build and elaborate on the ideas of
Ginosar and Holzman~\cite{ginosar2000majority} to apply these
techniques to infinite graphs. We apply them to both the asynchronous
and the synchronous model, and use them to prove some additional,
general results for majority dynamics on graphs with uniformly bounded
growth.

We begin by introducing some non-standard notation which will allow us
to simultaneously treat the synchronous and the asynchronous
models. We let the symbol $\Delta$ stand for ``$1$'' for in
synchronous model and for ``$dt$'' in the asynchronous model. For the
asynchronous model, we denote
\begin{align*}
  \op^i_{t - \Delta} = \op^i_{t - dt} = \lim_{t' \to t^-}\op^i_{t'}
\end{align*}
and
\begin{align*}
  \op^i_{t + \Delta} = \op^i_{t + dt} = \lim_{t' \to
    t^+}\op^i_{t'} =\op^i_t,
\end{align*}
where the last equality follows from Eq.~\ref{eq:maj-asynch}. The
definitions of majority dynamics, Eqs.~\ref{eq:maj-synch}
and~\ref{eq:maj-asynch}, can now be written in one equation:
\begin{align*}
  \op^i_{t+\Delta} = \sgn \sum_{j \in \neigh{i}}\op^j_t.
\end{align*}

Let $G=(V,E)$ be a finite or infinite $d$-bounded degree graph. Let a
{\em $d$-legal edge weighting} $z : E \to (0,1]$ be such that for any two
adjacent edges $e_1=(i,j)$ and $e_2=(i,k)$ it holds that
\begin{align}
  \label{eq:z}
  \frac{z(e_1)}{z(e_2)} < 1+\frac{2}{d-1} = \frac{d+1}{d-1}.
\end{align}
Note that $z$ is a function on the undirected edge set $E$, and so
$z(i,j) = z(j,i)$.  It is easy to see that
\begin{align}
  \label{eq:sign-z}
  \sgn \sum_{j \in \neigh{i}}\op^j_t = \sgn \sum_{j \in \neigh{i}}z(i,j)\op^j_t.
\end{align}
Indeed, assume $i$ has $k$ neighbors, and that the majority of their
opinions are, without loss of generality, $+1$. Then $i$ has at least
$\frac{k+1}{2}$ neighbors with opinion $+1$. Denote by $N_+$ this set
of neighbors. Similarly, $i$ has at most $\frac{k-1}{2}$ neighbors
with opinion $-1$, which we denote by $N_-$. Let $z_1$ be
the largest $z(i,j)$ for $j\in\neigh{i}$, and $z_2$ be the smallest
$z(i,j)$ in the same set. By assumption, and because $k \leq d$,
\[
\frac{z_1}{z_2}< \frac{d+1}{d-1} \leq \frac{k+1}{k-1}.
\]
Hence,
\begin{align*}
  \sum_{j \in \neigh{i}}z(i,j)\op^j_t&\geq z_2\sum_{j\in
    N_+}1-z_1\sum_{j\in N_-}1\\ &=
  z_2\left(|N_+|-\frac{z_1}{z_2}|N_-|\right)\\
  &\geq z_2\left(\frac{k+1}{2}-\frac{z_1}{z_2}\frac{k-1}{2}\right),
\end{align*}
the last expression is positive due to the second inequality in the
previous display, and Eq.~\ref{eq:sign-z} follows.  Thus, the
definitions of majority dynamics (Eqs.~\ref{eq:maj-synch}
and~\ref{eq:maj-asynch}) can equivalently be written as
\begin{align*}
  \op^i_{t+\Delta} = \sgn \sum_{j \in \neigh{i}}z(i,j)\op^j_t.
\end{align*}

Yet another equivalent definition is
\begin{align}
  \label{eq:argmin}
  \op^i_{t+\Delta} = \argmin_{a \in \{-1,+1\}}\sum_{j \in \neigh{i}}z(i,j)(\op^j_t-a)^2;
\end{align}
intuitively, each agent in each turn can be seen as trying to minimize
the ``energy'' $\sum_{j \in \neigh{i}}z(i,j)(\op^j_t-a)^2$. This
motivates the following definition:
\begin{align*}
  L_t = \fourth\sum_{(i,j) \in E}z(i,j)(\op^i_{t+\Delta}-\op^j_t)^2.
\end{align*}
Of course, $L_t$ has to be finite for this definition to be useful,
and we indeed give in Proposition~\ref{thm:summable} a necessary and
sufficient geometrical condition under which a $d$-legal summable $z$
exists. Note also that $L_t$ depends (implicitly) on the choice of
$z$.

It turns out that this definition of $L_t$ is a correct choice for a
``Lyapunov functional'', in the sense that $L_t$ is monotone
non-increasing.
\begin{proposition}
  \label{thm:lyapunov}
  $L_{T+t} \leq L_T$ for all $T,t \geq 0$.
\end{proposition}
Before proving this proposition, we will need the following
definitions and lemma.  Define $J^i_t$ by
\begin{align*}
  J^i_t = \half(\op^i_{t+\Delta}-\op^i_{t-\Delta})\sum_{j
    \in \neigh{i}}z(i,j)\op^j_t,
\end{align*}
and let
\begin{align*}
  J_t = \sum_{i \in V}J^i_t.
\end{align*}
\begin{claim}
  \label{thm:J1}
  $J^i_t \geq 0$, and $J^i_t = 0$ iff $\op^i_{t+\Delta}=\op^i_{t-\Delta}$.
\end{claim}
\begin{proof}
  Since we assume that $|\neigh{i}|$ is odd, $\sum_{j
    \in \neigh{i}}z(i,j)\op^j_t$ is never zero. It follows that
  $J^i_t = 0$ iff $\op^i_{t+\Delta}=\op^i_{t-\Delta}$.

  To see that $J^i_t \geq 0$, note that when
  $\op^i_{t+\Delta}=\op^i_{t-\Delta}$ then $J^i_t=0$. Otherwise we
  have that $J^i_t = \op^i_{t+\Delta}\sum_{j \in \neigh{
    i}}z(i,j)\op^j_t$. But $\op^i_{t+\Delta} = \sgn \sum_{j
    \in \neigh{i}}z(i,j)\op^j_t$ and so $J^i_t$ is equal to the
  product of two equal sign multiplicands and is therefore positive.
\end{proof}
\begin{claim}
  \label{thm:J2}
  \begin{align*}
    L_t-L_{t-\Delta} = -J_t.
  \end{align*}
\end{claim}
\begin{proof}
  \begin{align*}
    L_t-L_{t-\Delta} &= \fourth\sum_{(i,j) \in
      E}z(i,j)\left[(\op^i_{t+\Delta}-\op^j_t)^2-(\op^i_t-\op^j_{t-\Delta})^2\right]\\
    &= -\half\sum_{(i,j) \in
      E}z(i,j)\left[\op^i_{t+\Delta}\op^j_t-\op^i_t\op^j_{t-\Delta}\right]
  \end{align*}
  Since the edges are undirected, summing over $(i,j) \in E$ is the
  same as summing over $(j,i) \in E$. Therefore, and since
  $z(i,j)=z(j,i)$, we can exchange the roles of $i$ and $j$ in the
  last summand:
  \begin{align*}
    &= -\half\sum_{(i,j) \in
      E}z(i,j)\left[\op^i_{t+\Delta}\op^j_t-\op^j_t\op^i_{t-\Delta}\right]\\
    &= -\sum_{i \in V}\half(\op^i_{t+\Delta} - \op^i_{t-\Delta})\sum_{j
      \in \neigh{i}}z(i,j)\op^j_t\\
    &= -\sum_{i \in V}J^i_t\\
    &= -J_t.
  \end{align*}
\end{proof}
The proof of Proposition~\ref{thm:lyapunov} is now immediate.
\begin{proof}[Proof of Proposition~\ref{thm:lyapunov}]
  Since $L_t-L_{t-\Delta} = -J_t$ and since $J_t \geq 0$, it follows
  that $L_{T+t} \leq L_T$ for all $T,t \geq 0$.
\end{proof}

\subsection{No dynamic monopolies on infinite graphs}

As a simple application of Claims~\ref{thm:J1} and~\ref{thm:J2} we
show that infinite graphs cannot have dynamic monopolies.
\begin{proof}[Proof of Theorem~\ref{thm:no-dynamic-monopoly}]

  Consider the synchronous model.  Let $G$ be an infinite graph, and
  let $z : E \to \R^+$ be the $d$-legal edge weighting given by the
  constant function $z=1$. Then $z$ is not summable, but
  Claims~\ref{thm:J1} and~\ref{thm:J2} still hold, given that initial
  signals are chosen so that $L_t$ is finite.

  Let $W$ be a finite set of vertices. Let $\op^i_0=+1$ for all $i \in
  W$ and $\op^i_0=-1$ for all $i \notin W$. Then $L_0$ is finite. Note
  that with $z$ constant, $L_t$ and $J^i_t$ are integer, and therefore
  $L_t$ decreases by at least one whenever $\op^i_{t+1} \neq
  \op^i_{t-1}$. Since $L_t \geq 0$, at most a finite number of agents
  change their opinion to $+1$, and $W$ is not a dynamic monopoly.

\end{proof}

It may be possible to give a stronger, quantitative version of this
theorem, by using an edge weighting that increases with the distance
from $W$. 

\subsection{Summable edge weightings and slow growth}
We are now almost ready to prove Theorem~\ref{thm:moran}. Before that,
we will show that graphs with slow growth admit summable, $d$-legal edge
weightings.
\begin{proposition}
  \label{thm:summable}
  Let $G=(V,E)$ be a $d$-bounded degree graph. Then the following are
  equivalent.
  \begin{enumerate}
  \item $G$ admits a summable, $d$-legal edge weighting.
  \item $\moment{d}{G,i} < \infty$ for some ($\leftrightarrow$ all) $i
    \in V$.
  \end{enumerate}
\end{proposition}
\begin{proof}
  Let $G$ be a $d$-bounded degree graph, and denote $a = (d+1)/(d-1)$.
  For a node $i$ and an edge $e=(j,k)$, denote by $\dist(i,e) =
  \min\{\dist(i,j),\dist(i,k)\}$ the distance of $e$ from $i$.

  We first show that (1) implies (2). Let $G$ admit a summable, $d$-legal
  edge weighting $z$. Let $e_0=(i,j)$ an edge. By the definition of
  $d$-legal edge weightings (Eq.~\ref{eq:z}), given an edge $e$ such that
  $\dist(i,e)=r$, we have that
  \begin{align*}
    z(e) \geq z(e_0)a^{-r-1}.
  \end{align*}
  Hence, by the fact that $z$ is summable, it follows that
  \begin{align*}
    \infty &> \sum_{e \in E}z(e) = \sum_{r=0}^\infty\sum_{\{e \in E\,:\,
      \dist(i,e)=r\}}z(e) \geq \frac{z(e_0)}{a}\sum_{r=0}^\infty a^{-r}|\{e \in
    E\,:\,\dist(i,e) = r\}|.
  \end{align*}
  Now, the number of edges at distance $r$ is greater than or equal to
  the number of nodes at distance $r+1$, $n_{r+1}(G,i)$. Hence
  \begin{align*}
    \infty &> \frac{z(e_0)}{a}\sum_{r=0}^\infty n_{r+1}(G,i)a^{-r} \\
    &= z(e_0)\sum_{r=1}^\infty n_r(G,i)a^{-r} \\
    &= z(e_0)\left(\moment{d}{G,i}-1\right),
  \end{align*} 
  where the last equality follows from the definition of
  $\moment{d}{G,i}$ and the fact that $n_0(G,i)=1$. Thus
  $\moment{d}{G,i}$ is finite.

  We now show that (2) implies (1). Let $\moment{d}{G,i} < \infty$ for
  some $i \in V$.  Let $\eta : \N \to (1/a,1)$ be any monotone
  increasing function, and let
  \begin{align*}
    z(e) = a^{-\dist(i,e)}\frac{\eta(\dist(i,e))}{\eta(0)}
  \end{align*}
  be an edge weighting. We will show that it is $d$-legal and summable. Indeed,
  \begin{align*}
    \sum_{e \in E}z(e) &= \sum_{r=0}^\infty
    a^{-r}\frac{\eta(r)}{\eta(0)}|\{e \in
    E\,:\,\dist(i,e) = r\}|\\
    &\leq \sum_{r=0}^\infty a^{-r}\frac{\eta(r)}{\eta(0)}\cdot d\cdot|\{j \in
    V\,:\,\dist(i,j) = r\}|,
  \end{align*}
  since the number of edges at distance $r$ is at most $d$ times the
  number of vertices at that distance. Since $n_r(G,i) = |\{j \in
  V\,:\,\dist(i,j) = r\}|$ and since $\eta(r) < 1$ then
  \begin{align*}
    &< \frac{d}{\eta(0)}\sum_{r=0}^\infty a^{-r}\cdot n_r(G,i)\\
    &= \frac{d}{\eta(0)} \cdot \moment{d}{G,i}\\
    &< \infty,
  \end{align*}
  and so $z$ is summable.

  To see that $z$ is $d$-legal, note that if $e_1$ and $e_2$ are adjacent
  then either
  \[
  \dist(i,e_1) = \dist(i,e_2)
  \]
   in which case
  $z(e_1)=z(e_2)$, or else, without loss of generality
  \[\dist(i,e_2) = \dist(i,e_1) + 1.\]
  In this case, denoting $r=\dist(i,e_1)$,
  \begin{align*}
    \frac{z(e_2)}{z(e_1)}  =
    \frac{a^{-r-1}\eta(r+1)}{a^{-r}\eta(r)}
    =a^{-1}\frac{\eta(r+1)}{\eta(r)} < 1,
  \end{align*}
  where the last inequality follows from the fact that $\eta(r+1) < 1$
  and $\eta(r) > 1/a$. Likewise,
  \begin{align*}
    \frac{z(e_1)}{z(e_2)}  =
    \frac{a^{-r}\eta(r)}{a^{-r-1}\eta(r+1)}
    =a\frac{\eta(r)}{\eta(r+1)} < a = \frac{d+1}{d-1},
  \end{align*}
  where the last inequality follows from the fact that $\eta$ is
  monotone increasing.
\end{proof}

\begin{proof}[Proof of Theorem~\ref{thm:moran}]
  Let $z$ be a $d$-legal summable edge weighting, as defined in the proof
  of Proposition~\ref{thm:summable}. Note that $L_t$ is non-negative, and
  also finite for all $t$:
  \begin{align*}
    L_t &= \fourth\sum_{(i,j) \in
      E}z(i,j)(\op^i_{t+\Delta}-\op^j_t)^2\\
    &\leq \sum_{(i,j) \in E}z(i,j).
  \end{align*}

  Since $L_t$ is finite, non-negative and non-increasing, it follows that
  \begin{align*}
    L = \lim_{t \to \infty}L_t
  \end{align*}
  always exists and is non-negative.

  Fix the initial opinions, the times of the clock rings (for the
  asynchronous model) and a vertex $j$. Note that $J^j_t$ is either
  zero or greater than some $\eps_j > 0$. Let $T$ be such that $L_T-L
  < \eps_j$. It then follows from Claims~\ref{thm:J1} and~\ref{thm:J2}
  that $J^j_t = 0$ for all $t > T$, since otherwise it would be the
  case that $J^j_t > \eps_j$ and $L_{t+\Delta} < L$. Now, by
  Claim~\ref{thm:J1} it follows from this that
  $\op^j_{t+\Delta}=\op^j_{t-\Delta}$ for all $t > T$, and so $j$'s
  opinion has period at most two in the synchronous case, and
  converges in the asynchronous case.
\end{proof}

\subsection{Combinatorial majority dynamics results}
Let $G=(V,E)$ be a $d$-bounded degree graph such that $\moment{d}{G,i}
< \infty$ for some $i \in V$.  As per the proof of
Proposition~\ref{thm:summable}, let $z$ be a summable, $d$-legal edge
weighting given by
\begin{align*}
  z(e) = a^{-\dist(i,e)}\frac{\eta(\dist(i,e))}{\eta(0)},
\end{align*}
where $a = (d+1)/(d-1)$, and $\eta : \N \to (1/a,1)$ is some monotone
increasing function. In fact, by the proof of
Proposition~\ref{thm:summable}, we have that
\begin{align*}
  \sum_{e \in E}z(e) \leq \frac{d}{\eta(0)} \cdot \moment{d}{G,i} \leq
  \frac{d+1}{d-1}\cdot d \cdot \moment{d}{G,i}.
\end{align*}

\begin{proof}[Proof of Theorem~\ref{thm:opinion-changes}]
  By Claim~\ref{thm:J1}, the number of times that
  $\op^i_{t+\Delta} \neq \op^i_{t-\Delta}$ is equal to the
  number of times that $J^i_t \neq 0$. Hence we will prove the claim
  by bounding the number of times that $J^i_t \neq 0$.

  Since $z(i,j)=1$ for all $j \in \neigh{i}$, $J^i_t \geq 1$ whenever
  $J^i_t \neq 0$. Since $L=\lim_{t \to \infty}L_t$ is non-negative, and
  since $L_t-L_{t-\Delta} = -J_t$ by Claim~\ref{thm:J2}, it follows
  that the number of times that $J^i_t \neq 0$ is at most $L_0$.

  Now, in the asynchronous case
  \begin{align*}
    L_0 &= \fourth\sum_{(k,j) \in
      E}z(k,j)(\op^k_0-\op^j_0)^2\\
    &\leq 2\sum_{e \in E}z(e)\\
    &\leq \frac{d+1}{d-1} \cdot 2d \cdot \moment{d}{G,i}.
  \end{align*}

  In the synchronous case this bound can be improved. By
  Eq.~\ref{eq:argmin}, we have that
  \begin{align*}
    \fourth \sum_{j \in N(k)}z(j,k)(\op^k_1-\op^j_0)^2 < \half \sum_{j \in
      N(k)}z(j,k).
  \end{align*}
  Hence
  \begin{align*}
    L_0 &= \fourth\sum_{(k,j) \in
      E}z(k,j)(\op^k_1-\op^j_0)^2\\
    &< \sum_{e \in E}z(e)\\
    &\leq \frac{d+1}{d-1} \cdot d \cdot \moment{d}{G,i}.
  \end{align*}
\end{proof}

\begin{proof}[Proof of Theorem~\ref{thm:bunkers}]
  Let $r_0$ be such that
  \begin{align*}
     \frac{d+1}{d-1} \cdot 2d \cdot \sum_{r > r_0}\left(\frac{d-1}{d+1}\right)^r \cdot n_r(G,i) < 1.
  \end{align*}
  Then it is easy to verify that
  \begin{align*}
    \sum_{e\in E}\ind{\dist(i,e) > r_0}z(e) < \half.
  \end{align*}

  If $\op^j_T=s$ for all $j$ within distance $r_0+2$ from $i$, then
  $\op^j_{T+\Delta}=s$ for all $j$ within distance $r_0+1$ from $i$, and
  $\op^k_{T+\Delta}=s$ for all $e=(j,k)$ with $\dist(i,e) \leq
  r_0$. Hence
  \begin{align*}
    L_T &= \fourth\sum_{(k,j) \in
      E}z(k,j)(\op^k_{T+\Delta}-\op^j_T)^2\\
    &\leq 2\sum_{e\in E}\ind{\dist(i,e) > r_0}z(e)\\
    &< 1.
  \end{align*}

  Now, as in the proof of Theorem~\ref{thm:opinion-changes} above, if
  $J^i_t \neq 0$ then $J^i_t \geq 1$. However, since $L_T < 1$, since
  $L_t-L_{t-\Delta} = -J_t$, and since $L_t$ is non-negative, it
  follows that $J^i_t=0$ for all $t > T$, and so, in particular,
  $\op^i_{t+\Delta}=\op^i_{t-\Delta}$ for all $t > T$. This completes
  the proof for the asynchronous model. In the synchronous model,
  $\op^i_T=\op^i_{T+1}$, since all of $i$'s neighbors also have
  $\op^j_T=s$, and so it follows that $\op^i_T=s$ for all $t>T$.
\end{proof}

\subsection{Invariant random subgraphs}

\subsubsection{Light cones}
We commence by defining for each vertex $i$ and time $t$ the set of
vertices $\cc^i_t$ that form $i$'s {\em past light cone} at time
$t$. This is the set of vertices whose initial opinions may have
influenced $i$'s opinion at time $t$. Formally, set $\cc^i_0 = \{i\}$
for all $i \in V$. At every time $t$ in which $i$ updates its opinion,
update $\cc^i_t$ by
\begin{align*}
  \cc^i_t = \bigcup_{j \in \neigh{i}}\cc^j_{t-\Delta}.
\end{align*}
In the synchronous model, $\cc^i_t$ is simply the ball of radius $t$
around $i$. In the asynchronous model every vertex is a member of
$\cc^i_t$ with positive probability. However $\cc^i_t$ is still finite
with probability one.
\begin{claim}
  \label{clm:asynch-well-defined}
  Let $G$ be a $d$-bounded degree graph. Then for all $i \in V$ and $t
  \geq 0$ it holds that $\P{|\cc^i_t| < \infty} = 1$.
\end{claim}
It follows that the asynchronous model is well defined, since if there
exists an infinite sequence of times $t_1 > t_2 > \cdots > 0$ such
that $i$ updates at time $t_i$ and $i \in \neigh{i+1}$, then
$|\cc^1_{t_1}| = \infty$.
\begin{proof}
  Note that $|\cc^i_t|$ is stochastically dominated from above by the
  total number of offsprings in a Galton-Watson process, where each
  vertex has $d$ children with probability equal to the probability of
  a clock ring in $[0,t]$, and zero children otherwise. Since, for $t$
  small enough, the expected number of children in this process is
  lower than one, it follows that the total number of offsprings is
  almost surely finite. Hence for $t$ small enough (e.g., $1/(100d)$)
  we have that $\P{|\cc^i_t|=\infty}=0$.

  Now, the event that $|\cc^1_t| = \infty$ is equivalent to the
  existence infinite sequence of times $t_1 > t_2 > \cdots > 0$ such
  that (after an appropriate renaming of the vertices) $i$ updates at
  time $t_i$ and $i \in \neigh{i+1}$. By the time-shift-invariance of
  the clock tick process we can assume that $\lim_it_t=0$, and so, if
  $\P{|\cc^1_t| = \infty} > 0$ for some $t$ then
  $\P{|\cc^j_{1/(100d)}| = \infty} > 0$ for some $j$. But this is
  false by the above, and so $\P{|\cc^i_t| = \infty} = 0$ for all $i$
  and all $t$.
\end{proof}

We say that vertices $i$ and $j$ are {\em causally connected} at time
$t$ if $\cc^i_t \cap \cc^j_t \neq \emptyset$. The following is immediate.
\begin{claim}
  \label{clm:independent}
  Condition on the event that the clock ring times are such that $i$
  and $j$ are not causally connected at time $t$. Then the events
  `$\op^i_t=S$' and `$\op^j_t=S$' are independent.
\end{claim}

\subsubsection{Choosing independent vertices}
\label{sec:independent-vertices}
We next proceed to construct, for each time $t$, a set of vertices
whose opinions at time $t$ are close to being independent, conditioned
on $S$.

Let $G=(V,E)$ be an infinite quasi-transitive graph.  Then it is
standard to show that given $\delta > 0$ and a positive integer $t$,
one may find a number $r_{t,\delta}$ such that, if $i,j$ are two
vertices whose distance is more than $r_{t,\delta}$, then the
probability that $i$ and $j$ are causally connected at time $t$
(meaning $\cc_t^i \cap \cc_t^j \neq \emptyset$) is at most
$\delta$. Note that in the synchronous case one can take
$r_{t,\delta}=r_{t,0}=2t$.

For every time $t$ and $\delta>0$, let $W_{t,\delta} \subseteq V$ be a
random subset of $V$, drawn as follows. Associate to each vertex in $i
\in V$ an i.i.d.\ exponential random variable $X_i$.  A vertex $i$
belongs to $W_{t,\delta}$ if and only if $X_i < X_j$ for all $j$ with
$\dist(i,j) \leq r_{t,\delta}$.

We note a number of easily verifiable facts regarding $W_{t,\delta}$.
\begin{enumerate}
\item Any two elements of $W_{t,\delta}$ are at
  distance more than $r_{t,\delta}$. Hence, by the definition of causal
  connectedness, any two vertices in $W_{t,\delta}$ are causally
  connected with probability at most $\delta$.

\item For all $h \in \aut(G)$ it holds that $h(W_{t,\delta})$ has the same
  law as $W_{t,\delta}$; that is, the distribution of $W_{t,\delta}$ is
  $\aut(G)$-invariant.

\item $W_{t,\delta}$ is almost surely infinite, since in
  quasi-transitive graphs the size of radius $r$ balls is uniformly
  bounded.
  
  If furthermore $H \leq \aut(G)$ acts quasi-transitively on $G$ then
  $W_{t,\delta}$ intersects each orbit $V/H$ infinitely often.
\end{enumerate}

\subsubsection{Convergence and uniform convergence to $\limop$}
Let $G=(V,E)$ be a slow-growth graph. Then, as we show above, the
limit $\limop^i = \lim_t\op^i_{2t}$ exists almost surely. Hence for
each vertex $i$ the function $q_i$ given by $q_i(2t) = \P{\limop^i
  \neq \op^i_{2t}}$ converges to zero.

Clearly, if $i$ and $j$ are both in the same orbit $V / \aut(G)$, then
$q_i=q_j$. It follows that for every quasi-transitive graph, that is,
for every graph with a finite number of orbits, there exists a single
function $q$ that converges to zero and such that $q(2t) \geq
\P{\limop^i \neq \op^i_{2t}}$ for all $i \in V$. In this case we say
that we have uniform convergence to $\limop^i$. The next claim
states that the same holds on a shift-invariant random graph chosen
from a quasi-transitive graph.
\begin{claim}
  \label{clm:uniform-convergence}
  Let $G_0$ be a quasi-transitive infinite graph with maximal degree
  $d$ such that $\moment{d}{G_0,i} < \infty$ for some
  ($\leftrightarrow$ all) $i \in V_0$.

  Let $H \leq \aut(G_0)$ act quasi-transitively on $G_0$, and let $G$
  be an infinite random subgraph of $G_0$ with an $H$-invariant
  distribution. Then there exists a function $q : 2\N \to [0,1]$ with
  $\lim_nq(2n) = 0$ such that $q(2t) \geq \CondP{\limop^i \neq
    \op^i_{2t}}{i \in G}$.
\end{claim}
\begin{proof}
  Let $i,j \in V_0$ belong to the same orbit in $V/H$, so that there
  exists an $h \in H$ such that $h(i) = j$.

  Since $G$ is $h$-invariant, it is possible, using $h$, to couple two
  copies of our probability space in such a way that $\op^i_t$ in the
  first copy equals $\op^j_t$ in the second for all $t$, $\limop^i$
  in the first copy equals $\limop^j$ in the second, and furthermore
  $i \in G$ in the first copy iff $j \in G$ in the second. It follows
  that
  \begin{align*}
    \CondP{\limop^i \neq \op^i_{2t}}{i \in G} = \CondP{\limop^j
      \neq \op^j_{2t}}{j \in G}.
  \end{align*}
  Let $(i_1,\ldots,i_k)$ be representatives of the orbits
  $V/H$. Then
  \begin{align*}
    q(2t) = \max_k\{\CondP{\limop^{i_k} \neq \op^{i_k}_{2t}}{i_k \in
      G}\}
  \end{align*}
  satisfies the conditions of the claim.
\end{proof}

\subsubsection{Personal estimates of $S$}
Before proceeding to show that $S$ can be well estimated given
$\{\limop^i\}_{i \in V}$, we note that each agent's opinion is always
equal to $S$ with probability at least $p$. This is true, by
definition, at time $t=0$, and it may not be surprising that this is
also the case in later times. However, the proof of this fact is not
completely straightforward.

We show in Claim~\ref{clm:asynch-well-defined} above that in the
asynchronous case, with probability one there does not exist an
infinite sequence of times $t_1 > t_2 > \cdots > 0$ such that $i$
updates at time $t_i$ and $i \in \neigh{i+1}$. Let $\bar{T}$ denote an
arbitrary choice of clock ring times for which indeed such a chain
does not exist. In the synchronous case let $\bar{T}$ be a trivial
(probability one) event.
\begin{claim}
  \label{clm:monotone-learning}
  For any $\bar{T}$, agent $i$ and time $t$ it holds that
  \begin{align*}
    \CondP{\op^i_t=S}{\bar{T}} \geq p.
  \end{align*}
\end{claim}
\begin{proof}
  Conditioned on $\bar{T}$, $\cc^i_t$ is fixed and so $\op^i_t$ is a
  deterministic function of $\{\op^j_0\,:\,j \in \cc^i_t\}$. Let $k =
  |\cc^i_t|$, and denote this function by $f:\{-1,+1\}^k \to
  \{-1,+1\}$.

  Clearly, $f$ is monotone, in the sense that
  \begin{align*}
    f(x_1,\ldots,x_{i-1},x_i,x_{i+1},\ldots,x_k) \leq
    f(x_1,\ldots,x_{i-1},+1,x_{i+1},\ldots,x_k),
  \end{align*}
  and is symmetric in the sense that
  \begin{align*}
    f(-x_1,-x_2,\ldots,-x_k) = -f(x_1,x_2,\ldots,x_k).
  \end{align*}
  It then follows from Lemma 6.1 in~\cite{mossel2012majority} that
  \begin{align*}
    \CondP{\op^i_t = +1}{\bar{T},S=+1} \geq p,
  \end{align*}
  and so, unconditioned on $S$, it holds that
  \begin{align*}
    \CondP{\op^i_t = S}{\bar{T}} \geq p.
  \end{align*}
\end{proof}

\subsubsection{Estimating $S$}

As in the setting of Theorem~\ref{thm:percolation-learning}, let $G_0$
be a quasi-transitive infinite graph with maximal degree $d$ such that
$\moment{d}{G_0,i} < \infty$ for some ($\leftrightarrow$ all) $i \in
V_0$. Pick the sets $W_{t,\delta}$ from $V_0$, as described above in
Section~\ref{sec:independent-vertices}.

Let $H \leq \aut(G_0)$ act transitively on $G_0$, and let $G=(V,E)$ be
an infinite connected random subgraph of $G_0$ with an $H$-invariant
distribution.  Let $V_{t,\delta} = W_{t,\delta} \cap V$. Note that
$W_{t,\delta}$ and $V$ are independent and are both infinite and
$H$-invariant. Furthermore, as we note above, $W_{t,\delta}$
intersects every orbit $V/ \aut(G_0)$ infinitely. It follows that
$V_{t,\delta}$ is almost surely infinite.
\begin{claim}
  \label{clm:cesaro}
  Fix an enumeration of the vertices of $G_0$, let
  $\{i_1,i_2,\ldots\}$ be the induced enumeration of the vertices in
  $V_{t,\delta}$, and let
  \begin{align*}
    \Sest_{t,\delta}^n = \sgn\frac{1}{n}\sum_{k=1}^n\op^{i_k}_t.
  \end{align*}
  Then
  \begin{align*}
    \P{\Sest_{t,\delta}^n \neq S} \leq e^{-\frac{(p-1/2)^2}{2p}n} + n^2\delta.
  \end{align*}
\end{claim}
\begin{proof}
  Let $E$ be the event that the clock ring times are such that
  $\{i_1,\ldots,i_n\}$ are not causally connected at time $t$. This
  happens with probability at least $1-n^2\delta$. Then conditioned on
  $E$, by Claim~\ref{clm:independent}, the events `$\op^{i_k}_t=S$'
  are independent.
  
  Now, by Claim~\ref{clm:monotone-learning}, for all $i$ and $t$ it
  holds that $\CondP{\op^i_t=S}{E} \geq p$. Hence the claim follows by
  the Chernoff bound.
\end{proof}

We are now ready to prove Theorem~\ref{thm:percolation-learning}
\begin{proof}[Proof of Theorem~\ref{thm:percolation-learning}]
  We will prove the theorem by showing that for every $\eps>0$ there
  exists an $\Sest_\eps$ that is $\sigmainf$-measurable and such that
  \begin{align*}
    \P{\Sest_\eps \neq S} \leq \eps.
  \end{align*}

  Let $N$ and $\delta$ be such that for every $t$ it holds that
  \begin{align}
    \label{eq:low-err}
    \P{\Sest_{t,\delta}^N \neq S} \leq \eps/2,
  \end{align}
  where, as in Claim~\ref{clm:cesaro}, $\Sest_{t,\delta}^N$ is given
  by
  \begin{align*}
    \Sest_{t,\delta}^N = \sgn\frac{1}{N}\sum_{k=1}^N\op^{i_k}_t,
  \end{align*}
  with $\{i_1,i_2,\ldots\}$ an enumeration of the vertices in
  $V_{t,\delta}$ which is induced by a fixed enumeration of the
  vertices of $G_0$. The existence of such $N$ and $\delta$ satisfying
  Eq.~\ref{eq:low-err} is guaranteed by Claim~\ref{clm:cesaro}.

  Let $q$ be a function which satisfies
  \begin{align*}
    q(2t) \geq \P{\limop^i \neq \op^i_{2t}},~~ \lim_{t \to
      \infty}q(t)=0,
  \end{align*}
  as given in
  Claim~\ref{clm:uniform-convergence}. Let $T \in 2\N$ be such that
  $q(T) \leq \eps/(2N)$, and define $\Sest_\eps$ by
  \begin{align*}
    \Sest_\eps = \sgn\frac{1}{N}\sum_{k=1}^N\limop^{i_k},
  \end{align*}
  where $\{i_1,i_2,\ldots\}$ is the same enumeration used to define
  $\Sest_{T,\delta}^N$. Then
  \begin{align*}
    \P{\Sest_\eps \neq \Sest_{T,\delta}^N} \leq \P{\limop^{i_k} \neq
      \op^{i_k}_T\mbox{ for some $1 \leq k \leq N$}}.
  \end{align*}
  Now, 
  \begin{align*}
    \CondP{\limop^i \neq \op^i_T}{i \in V_{t,\delta}} =
    \P{\limop^i \neq \op^i_T} \leq q(T).
  \end{align*}
  The equality holds since the choice of $V_{t,\delta}$ is
  independent of the majority dynamics process, and the inequality is
  simply a reference to the definition of $q(\cdot)$.
  
  Hence by the union bound we have that
  \begin{align*}
    \P{\limop^{i_k} \neq \op^{i_k}_T\mbox{ for some $1 \leq k \leq N$}} \leq N
    \cdot q(T) \leq \eps / 2,
  \end{align*}
  and so
  \begin{align*}
    \P{\Sest_\eps \neq \Sest_{T,\delta}^N} \leq \eps/2.
  \end{align*}
  Combining this with Eq.~\ref{eq:low-err}, we have that
  \begin{align*}
    \P{\Sest_\eps \neq S} \leq \eps.
  \end{align*}

\end{proof}

\subsection{Families of uniformly bounded growth graphs}
Recall that given $d \geq 3$ and a function $f : \N \to \N$,
$\bdd(f,d)$ is the family of $d$-bounded degree graphs $G$ such that
$n_r(G,i) \leq f(r)$ for all vertices $i$ in $G$ and $r \in
\N$. Recall also that we say that $f : \N \to \N$ has {\em slow
  growth} if $\moment{d}{f} < \infty$.

Fix $d \geq 3$, and let $f$ have slow growth. Let $r_0$ be the
smallest integer for which it holds that
\begin{align*}
  \frac{d+1}{d-1} \cdot 2d \cdot \sum_{r >
    r_0}\left(\frac{d-1}{d+1}\right)^r \cdot f(r) < 1.
\end{align*}
That is, $r_0$ is the minimal number that satisfies the condition of
Theorem~\ref{thm:bunkers}.  Let $B_0(G,i)$ denote the set of vertices
at distance at most $r_0+2$ from $i$ in $G$, and note that
$|B_0(G,i)|$ is uniformly bounded for all graphs $G \in \bdd(f,d)$.

\begin{proof}[Proof of Theorem~\ref{thm:bdd-retention}]
  Denote by $U_i$ the event that $\op^j_0=S$ for all $j \in
  B_0(G,i)$. Note that by Theorem~\ref{thm:bunkers}, the event
  `$\limop^i=S$' contains $U_i$.

  Let $p_0 < 1$ be close enough to $1$ so that, for any vertex $i$ in
  any graph $G \in \bdd(f,d)$, and any $p > p_0$, it holds that
  \begin{align*}
    \P{U_i} \geq \eta > \half
  \end{align*}
  for some $\eta$.  This holds, for example, for $p_0 > 1-1/(2N)$,
  where $N$ is a uniform bound on $|B_0(G,i)|$.

  Given a graph $G \in \bdd(f,d)$, let $I_G \subset V$ be a maximal
  set of vertices such that $B_0(G,i) \cap B_0(G,j) = \emptyset$ for
  all $i,j \in I_G$. Then the events $\{U_i\}_{i \in I_G}$ are
  independent, since the events `$\op^i_0=S$' are independent.

  We first consider the case of finite graphs. Let $\{G_n\}_{n \in
    \N}$ in $\bdd(f,d)$ be a sequence of graphs such that $\lim_n|V_n|
  = \infty$, and let $I_n = I_{G_n}$. Note that
  $\lim_{n\to\infty}|I_n| = \infty$, since the diameters of $G_n$ tend
  to infinity.

  Since $\P{U_i} \geq \eta$, and since the events $U_i$ are independent,
  then by the Chernoff bound we have that
  \begin{align*}
    \P{\sum_{i \in I_n}\ind{U_i} < |I_n|/2} \leq e^{-\frac{(\eta-1/2)^2}{2\eta}|I_n|},
  \end{align*}
  and so, since the event `$\limop^i=S$' contains $U_i$, we have that
  \begin{align*}
    \P{\Sest_n \neq S} \leq e^{-\frac{(\eta-1/2)^2}{2\eta}|I_n|},
  \end{align*}
  for
  \begin{align*}
    \Sest_n = \sgn\sum_{i \in I_n}\limop^i.
  \end{align*}
  Since $\lim_{n \to \infty}|I_n| = \infty$ it follows that
  \begin{align*}
    \lim_n\delta(G_n,p) = 0,
  \end{align*}
  and we have proved the claim for finite graphs.

  Let $G=(V,E) \in \bdd(f,d)$ be an infinite graph. Then, since the
  diameter of $G$ is infinite, there exists for every $n$ a set $I_n
  \subset V$ of size $n$ such that $B_0(G,i) \cap B_0(G,j) =
  \emptyset$ for all $i,j \in I_n$. By the same reasoning as above, we
  have that
  \begin{align*}
    \P{\Sest_n \neq S} \leq e^{-\frac{(\eta-1/2)^2}{2\eta}n},
  \end{align*}
  for
  \begin{align*}
    \Sest_n = \sgn\sum_{i \in I_n}\limop^i,
  \end{align*}
  and so
  \begin{align*}
    \delta(G,p) = 0.
  \end{align*}
\end{proof}

\subsection{Convergence on a non-slow growth graph}
\label{sec:counterexample}

In this section we provide an example of a $d$-regular graph for which
$\moment{d}{G,i} = \infty$ for all $i \in V$, but still the opinion of
every agent converges, in the synchronous model, to period at most
two, for every set of initial opinions.

Given an odd $d\geq 3$, denote by $H$ the complete bipartite graph
$K_{d,d}$ on $2d$ vertices, without a single edge $\{a,b\}$.  Let $G'$
be any $d-$regular graph. Let $F$ be any set of edges whose removal
separates $G'$ into finite components. Consider the following
process. For an edge $e=\{i,j\}\in F$, remove $e$ from $G'$, add a
copy $H_e$ of $H$, connect $i$ to $a_e$ (the vertex in $H_e$ which
corresponds to $a$), and connect $j$ to $b_e$. Repeat this process for
all $e\in F$. The resulting graph will be denoted by $G$.
\begin{lemma}
  For any initial choice of opinions for $G$, the opinions of all
  vertices converge, in the synchronous model, to period at most two.
\end{lemma}
Since $G'$ was arbitrary, one can easily construct such graphs $G$
with $\moment{d}{G,i} = \infty$.
\begin{proof}
  We first show that for any copy of $H$ the opinions converge at time
  at most $t=2$ to a period at most two. Fix a copy of $H$, call it
  also $H$.

  Indeed, if we call one side of $H$ $A$ and the other $B$, so that
  $a\in A,~b\in B$, then at time $t=1$ all vertices of
  $A\setminus\{a\}$ have value $m_B$, the majority of values in $B$,
  and similarly vertices in $B\setminus\{b\}$ have value $m_A$. As
  $d\geq 3$, both $a$ and $b$ have more neighbors in $H$ than in
  $G\setminus{H}$. Thus, at time $t=2$ all $A$ have opinions $m_A$,
  and all $B$ have $m_B$. Now it is easy to see that for this subgraph
  we have converged to a period of length at most two.

  Since when we remove the copies of $H$ from $G$ we get a union of
  finite components, after time $t=2$, the processes in these
  components are independent and can be thought as applying the
  dynamics to finite graphs, with some fixed boundary
  conditions. Thus, the opinions in these components will also
  converge to a period of length at most two~\cite{goles1980periodic}.

\end{proof}

\bibliography{retention}

\begin{bibdiv}
\begin{biblist}

\bib{bala1998learning}{article}{
      author={Bala, Venkatesh},
      author={Goyal, Sanjeev},
       title={Learning from neighbours},
        date={1998},
     journal={The Review of Economic Studies},
      volume={65},
      number={3},
       pages={595\ndash 621},
}

\bib{berger2001dynamic}{article}{
      author={Berger, Eli},
       title={Dynamic monopolies of constant size},
        date={2001},
     journal={Journal of Combinatorial Theory, Series B},
      volume={83},
      number={2},
       pages={191\ndash 200},
}

\bib{camia2002clusters}{article}{
      author={Camia, Federico},
      author={De~Santis, Emilio},
      author={Newman, Charles},
       title={Clusters and recurrence in the two-dimensional zero-temperature
  stochastic {I}sing model},
        date={2002},
     journal={The Annals of Applied Probability},
      volume={12},
      number={2},
       pages={565\ndash 580},
}

\bib{de2003convergence}{article}{
      author={De~Santis, Emilio},
      author={Newman, Charles},
       title={Convergence in energy-lowering (disordered) stochastic spin
  systems},
        date={2003},
     journal={Journal of statistical physics},
      volume={110},
      number={1-2},
       pages={431\ndash 442},
}

\bib{ellison1993rules}{article}{
      author={Ellison, Glenn},
      author={Fudenberg, Drew},
       title={Rules of thumb for social learning},
        date={1993},
     journal={Journal of Political Economy},
       pages={612\ndash 643},
}

\bib{fontes2002stretched}{article}{
      author={Fontes, Luiz~Renato},
      author={Schonmann, Roerto},
      author={Sidoravicius, Vladas},
       title={Stretched exponential fixation in stochastic {I}sing models at
  zero temperature},
        date={2002},
     journal={Communications in Mathematical Physics},
      volume={228},
      number={3},
       pages={495\ndash 518},
}

\bib{gale2003bayesian}{article}{
      author={Gale, Douglas},
      author={Kariv, Shachar},
       title={Bayesian learning in social networks},
        date={2003},
     journal={Games and Economic Behavior},
      volume={45},
      number={2},
       pages={329\ndash 346},
}

\bib{ginosar2000majority}{article}{
      author={Ginosar, Yuval},
      author={Holzman, Ron},
       title={The majority action on infinite graphs: strings and puppets},
        date={2000},
     journal={Discrete Mathematics},
      volume={215},
      number={1-3},
       pages={59\ndash 72},
}

\bib{goles1980periodic}{article}{
      author={Goles, Eric},
      author={Olivos, Jorge},
       title={Periodic behaviour of generalized threshold functions},
        date={1980},
        ISSN={0012-365X},
     journal={Discrete Mathematics},
      volume={30},
      number={2},
       pages={187 \ndash  189},
  url={http://www.sciencedirect.com/science/article/pii/0012365X80901211},
}

\bib{howard2000zero}{article}{
      author={Howard, C.~Douglas},
       title={Zero-temperature {I}sing spin dynamics on the homogeneous tree of
  degree three},
        date={2000},
     journal={Journal of Applied Probability},
       pages={736\ndash 747},
}

\bib{kanoria2011majority}{article}{
      author={Kanoria, Yashodhan},
      author={Montanari, Andrea},
       title={Majority dynamics on trees and the dynamic cavity method},
        date={2011},
     journal={The Annals of Applied Probability},
      volume={21},
      number={5},
       pages={1694\ndash 1748},
}

\bib{kempe2003gossip}{inproceedings}{
      author={Kempe, David},
      author={Dobra, Alin},
      author={Gehrke, Johannes},
       title={Gossip-based computation of aggregate information},
organization={IEEE},
        date={2003},
   booktitle={Proceedings of the 44th annual symposium on foundations of
  computer science},
       pages={482\ndash 491},
}

\bib{mcculloch1943logical}{article}{
      author={McCulloch, Warren},
      author={Pitts, Walter},
       title={A logical calculus of the ideas immanent in nervous activity},
        date={1943},
     journal={The Bulletin of Mathematical Biophysics},
      volume={5},
      number={4},
       pages={115\ndash 133},
}

\bib{moran1995period}{article}{
      author={Moran, Gadi},
       title={On the period-two-property of the majority operator in infinite
  graphs},
        date={1995},
     journal={Transactions of the American Mathematical Society},
      volume={347},
      number={5},
       pages={1649\ndash 1667},
}

\bib{mossel2012majority}{article}{
      author={Mossel, Elchanan},
      author={Neeman, Joe},
      author={Tamuz, Omer},
       title={Majority dynamics and aggregation of information in social
  networks},
        date={2014},
     journal={Autonomous Agents and Multi-Agent Systems},
      volume={28},
      number={3},
       pages={408\ndash 429},
}

\bib{mossel2012strategic}{article}{
      author={Mossel, Elchanan},
      author={Sly, Allan},
      author={Tamuz, Omer},
       title={Strategic learning and the topology of social networks},
        date={2012-08},
     journal={ArXiv e-prints},
      eprint={1209.5527},
}

\bib{peleg1998size}{article}{
      author={Peleg, David},
       title={Size bounds for dynamic monopolies},
        date={1998},
     journal={Discrete Applied Mathematics},
      volume={86},
      number={2},
       pages={263\ndash 273},
}

\bib{shah2009gossip}{article}{
      author={Shah, Devavrat},
       title={Gossip algorithms},
        date={2009},
     journal={Foundations and Trends in Networking},
      volume={3},
      number={1},
       pages={1\ndash 125},
}

\bib{louidortessler2010}{article}{
      author={{Tessler}, Ran},
      author={{Louidor}, Oren},
       title={{Geometry and Dynamics in Zero Temperature Statistical Mechanics
  Models}},
        date={2010-08},
     journal={ArXiv e-prints},
      eprint={1008.5279},
}

\end{biblist}
\end{bibdiv}
\end{document}